\documentclass[letterpaper,11pt,reqno]{amsart}
\usepackage[utf8]{inputenc}
\usepackage{amsmath}
\usepackage{amsthm}
\usepackage{amssymb}
\usepackage[colorlinks=true, linkcolor=blue, citecolor=blue, urlcolor=blue]{hyperref}
\usepackage{amsfonts,mathrsfs}
\usepackage{stmaryrd}
\usepackage{amsxtra}  
\usepackage{epsfig}
\usepackage{verbatim}
\usepackage{enumerate}
\usepackage{xcolor}
\usepackage{enumitem}
\usepackage{bigints}
\usepackage{mathrsfs}
\usepackage{tikz-cd}
\usepackage{bm,bbm}
\usepackage{hyperref}
\usepackage[all]{xy}
\usepackage{mathtools, hyperref}
\usepackage{tikz}
\usepackage{tikz-cd}
\usetikzlibrary{shapes}
\usetikzlibrary{plotmarks}
\usepackage{mathabx} 
\usepackage{float}
\usepackage{orcidlink}

\theoremstyle{plain}
\newtheorem{theorem}{Theorem}[section]
\newtheorem{proposition}[theorem]{Proposition}
\newtheorem{lemma}[theorem]{Lemma}
\newtheorem{corollary}[theorem]{Corollary}

\theoremstyle{definition}

\theoremstyle{remark}
\newtheorem{exmp}[theorem]{Example}
\newtheorem{remark}[theorem]{Remark}

\newcommand{\abs}[1]{\left\vert#1\right\vert}

\newcommand{\ol}{\overline}

\newcommand{\id}{\mathrm{id}}

\newcommand{\ipr}[1]{\left\langle #1 \right\rangle}

\newcommand{\xdownarrow}[1]{%
  {\left\downarrow\vbox to #1{}\right.\kern-\nulldelimiterspace}
}

\newcommand{\cx}{{\mathbb C}}

\newcommand{\C}{{\mathbb C}}

\title[The $L^p$- regularity for the Bergman projection of Rudin ball quotients in $\mathbb C^2$]{The $L^p$- regularity problem  for the Bergman projection of two-dimensional Rudin ball quotients}
\thanks{A.M. was partially supported by the European Union-NextGenerationEU, under the National
Recovery and Resilience Plan (NRRP), Mission 4, Component 2, Investment 1.1, funding call PRIN
2022 D.D. 104 published on 2.2.2022 by the Italian Ministry of University and Research (Ministero
dell'Universit\`a e della Ricerca), Project Title: TIGRECO-TIme-varying signals on Graphs: REal and COm-plex methods-CUP F53D23002630001. A.M is also member of Indam-GNAMPA. Part of this work was carried out during A.M.’s visit to Central Michigan University in March 2025, supported by the University of Bergamo Outgoing Program. A.M. gratefully acknowledges the warm hospitality of the Department of Mathematics at Central Michigan University}

\author{Debraj Chakrabarti}
\address{Department of Mathematics, Central Michigan University, Mt. Pleasant, MI 48859, USA}
\email{chakr2d@cmich.edu}

\author[A. Monguzzi]{Alessandro Monguzzi\orcidlink{0000-0003-3233-5000}}
\address{Dipartimento di Ingegneria Gestionale, dell’Informazione e della Produzione, Universit\`a
degli Studi di Bergamo, Viale G. Marconi 5, 24044, Dalmine BG, Italy}
\email{alessandro.monguzzi@unibg.it}

\begin{document}

\begin{abstract}
    We solve the $L^p$-regularity problem of the Bergman projection of two-dimensional Rudin ball quotients. 
\end{abstract}
\maketitle
\section{Introduction and Main Results}
Let $D$ be a domain in $\mathbb C^n$, i.e., a non-empty, open and connected subset of $\mathbb C^n$.
The Bergman space of $D$ is defined as
\[
A^2(D) := L^2(D) \cap \mathcal{O}(D),
\]
where $L^2(D)$ denotes the space of square-integrable functions on $D$ (with respect to the Lebesgue measure), and $\mathcal{O}(D)$ is the space of holomorphic functions on $D$. It is well known that $A^2(D)$ is a closed subspace of $L^2(D)$.

The orthogonal projection $P_D:L^2(D)\to A^2(D)$ is called the \emph{Bergman projection} associated with $D$ and can be represented as an integral operator with kernel $K_D(z,w)$, the \emph{Bergman kernel} of $D$.  

While $P_D$ is bounded on $L^2(D)$ by construction, understanding its behavior on other $L^p$ spaces is a subtle and challenging problem. More precisely, one seeks to determine for which $p \in [1,+\infty]$ the estimate  
\begin{equation}\label{eq:basic_estimate}
\|P_D f\|_{L^p} \leq C \|f\|_{L^p} \qquad \forall\, f \in L^2(D) \cap L^p(D)
\end{equation}
holds with a constant $C>0$. This question is often referred to as the \emph{$L^p$-regularity problem for the Bergman projection}, and its resolution depends sensitively on the geometry of the domain $D$. For instance, it was proved in \cite{D_L1} that \eqref{eq:basic_estimate} fails for \( p = 1 \) whenever \( D \) is a smooth domain. However, it is not known if the same result is true for a generic bounded domain. Classical results reveal a highly variable picture: for various classes of smooth pseudoconvex domains \cite{forelli_rudin, phong_stein, NRSW, mcneal_stein, mcneal_singular, charpentier_dupain, lanzani_stein_minimal}, the range of $p$ is known to be the whole interval  $(1,\infty)$. However, there are many examples where \( p \) is restricted to a proper subinterval of \( (1,\infty) \) \cite{krantz_peloso_houston, edholm_pacific, edholm_mcneal, monguzzi}, or even to the trivial case \( p = 2 \) only \cite{barrett_irregularity, zeytuncu, krantz_peloso_stoppato, EM3, krantz_monguzzi_peloso_stoppato}. For a detailed historical account and further references, see the recent survey \cite{zeytuncu_survey}.

Despite the extensive literature on the subject, the precise relationship between the geometry of the domain and the regularity of the Bergman projection remains poorly understood. In recent years, this question has been investigated in increasingly general settings, in the hope of uncovering deeper structural principles governing its $L^p$ regularity. In this paper we are interested to the setting of certain \emph{quotient domains}, equivalently, to certain classes of singular domains that can be covered, in the sense of ramified coverings, by simpler domains, e.g., by polydiscs or balls. Progresses for the $L^p$-regularity problem for some families of such domains have been made in  \cite{chen_krantz_yuan_polydisc, bender_chakrabarti_edholm_mainkar, DM, CJY, HW, LW, JLQW}. Other related papers worth mentioning are \cite{A, GG, CCGGW, EXX, EGX,  Ghosh_JMAA, QMZ}. Here we focus in particular on some domains obtained by taking the quotient of the unit ball $B_2$ of $\C^2$ with respect to the action of a finite unitary reflection group. These domains are sometimes called \emph{Rudin ball quotients} because of the seminal paper \cite{rudin_refl}.

Throughout the paper we use the notation $c$ to denote a positive constant, whose value may change from one occurrence to another.

\subsection{Rudin ball quotients} Let $G$ be a finite unitary reflection group (from now on, f.u.r.g.), that is, a finite subgroup of the unitary group $U(n)$ generated by reflections. A unitary reflection is an element of $U(n)$ of finite order that fixes pointwise exactly a hyperplane. We refer to \cite{lehrer_taylor} for the theory of f.u.r.g.'s. These groups are characterized among finite subgroups of $U(n)$ by the property that their ring of invariants,
\[
\C[z_1,\ldots,z_n]^G=\big\{P\in \C[z_1,\ldots,z_n]: P(g.z)=P(z)\quad \forall g\in G\big\},
\]
is generated by $n$ algebraically independent homogeneous polynomials $P_1,\ldots P_n$. Here, and from now on, we use the notation $g.z$ to denote the action of $g$ in $G$ on $\C^2$. 

For any f.u.r.g $G$ and any choice of generators $P_1,\ldots, P_n$ as above we define a  \emph{$G$-orbit map} $\pi:\C^n\to\C^n$ as
\[
\pi(z)=\big(P_1(z),\ldots P_n(z)\big), \quad z\in\C^n.
\]
Such maps have the following significant property.
\begin{proposition}[{\cite[Theorem 1.6]{rudin_refl}}, {\cite[Proposition 1.7]{DM}}
] Let $G$ be a f.u.r.g and let $\pi$ be a $G$-orbit map as above. Let $B_n$ be the unit ball of $\C^n$. Then, $D:=\pi(B_n)$ is a bounded domain and $\pi: B_n\to D$ is a proper holomorphic mapping.     
\end{proposition}
Rudin actually proved in \cite{rudin_refl} that any proper holomorphic map defined on the unit ball and satisfying a mild boundary regularity condition is, up to biholomorphisms of the domain and of the target, equivalent to a $G$-orbit map. Moreover, any \emph{proper map} $\pi:B_n\to D$ actually defines a \emph{ramified covering} (see, e.g., \cite[Remark 1.2]{DM}). Recall that $\pi$ is a proper map if $\pi^{-1}(K)$ is compact whenever $K\subseteq D$  is compact, whereas by ramified covering we mean that there exist two proper complex analytic subsets $V\subseteq B_n$ and $W\subseteq D$ such that $\pi: B_n\backslash V\to D\backslash W$ is a finite-sheeted covering map in the usual sense (see, e.g., Chapter III of \cite{bredon_book}) that is also holomorphic. In this case we say that $D$ is covered by the ball $B_n$ via the map $\pi$. It is clear that a $G$-orbit map $\pi$ is constant on the orbit of $G$ acting on $B_n$, hence we briefly write $\pi(B_n)=D=B_n/G$. From now on, when we refer to a \emph{Rudin ball quotient}, we mean a domain $D$ constructed in this manner.

To every ramified covering $\pi : B_n \to D$ we associate its \emph{covering group}  
\[
G = \big\{ g \in \mathrm{Aut}(B_n \setminus V) \ :\ \pi(g.z) = \pi(z) \ \ \forall z \in B_n \setminus V \big\}.
\]
The covering is called \emph{normal} if \(G\) acts transitively on the fibers of \(\pi\), that is,
\[
\pi(z_1) = \pi(z_2) \quad \Longleftrightarrow \quad \exists\, g \in G \ \text{with} \ g.z_1 = z_2.
\]
In particular, if $\pi:B_n\to D$ is a $G$-orbit map, then it is a \emph{normal ramified covering with covering group $G$} (\cite[Proposition 1.7]{DM}).

We are interested in the $L^p$-regularity of the Bergman projection on $D = \pi(B_n)$. Our approach will be to lift the $L^p$-regularity problem along a ramified covering, as it will be explained in the next section. Such an approach is not new, but it was first pioneered in \cite{bender_chakrabarti_edholm_mainkar} and later taken up again in \cite{DM}. We conclude the present section with an example of an infinite family of finite unitary reflection groups..

\begin{exmp}[{ \cite[Chapter 2]{lehrer_taylor}}] \label{ex:G} Let $m, n\geq1$ and assume that $\ell$ is a positive divisor of $m$.  Let $\theta$ be a primitive $m$-th root of unity. Then, the f.u.r.g. $G(m,\ell,n)$ consists of the linear transformations of the form
\[
(z_1,\ldots,z_n)\mapsto (\theta^{\nu_1}z_{\tau(1)},\ldots, \theta^{\nu_n}z_{\tau(n)})
\]
where $\tau$ is a permutation of $\{1,\ldots, n\}$ and the $\nu_j$'s are integers whose sum is divisible by $\ell$.

An orbit map for $G(m,\ell,2)$ is given by
\begin{equation}\label{orbit_map_G}
\pi(z_1,z_2)=\big(z_1^m+z_2^m, (z_1 z_2)^{\frac m\ell}\big).
\end{equation}
\end{exmp}

\subsection{Ramified coverings and $L^p$-regularity}
We now state a result illustrating how the $L^p$-regularity problem can be lifted along a ramified covering. This method was first successfully introduced in \cite{bender_chakrabarti_edholm_mainkar} in the study of $L^p$-regularity for monomial polyhedra, a class of domains admitting a ramified covering by polydiscs. It was subsequently employed in \cite{DM} with applications to the $L^p$-regularity problem for a particular family of Rudin ball quotients.

We present the result specifically for Rudin ball quotients, although it holds in greater generality. No proofs are provided here; for full details, we refer the reader to \cite[Section 4]{bender_chakrabarti_edholm_mainkar} and \cite[Section 2]{DM}.
The result is a consequence of a theorem of S.~Bell \cite[Theorem~1]{bell_proper}.

\begin{proposition}\label{prop: Lp_regularity_equivalence}

Let $G$ be a f.u.r.g., let $\pi$ be a $G$-orbit map, and set $D = \pi(B_n)$. Fix $p \in [1, +\infty)$ and $C > 0$. Then the estimate
\begin{equation*}
\int_D |P_D u|^p \leq C \int_D |u|^p \qquad \forall\, u \in L^2(D) \cap L^p(D)
\end{equation*}
is equivalent to the weighted estimate
\begin{equation}\label{Q_estimate}
\int_{B_n} |Q_{B_n,G} f|^p \,\sigma \leq C \int_{B_n} |f|^p \,\sigma
\qquad \forall\, f \in L^2(B_n) \cap L^p(B_n, \sigma),
\end{equation}
where $\sigma = |J(\pi)|^{2-p}$.  
Here $Q_{B_n,G}$ is the integral operator with kernel
\begin{equation}\label{averaged_kernel}
K_G(z,w) = \frac{1}{|G|} \sum_{g \in G} K_{B_n}(z, g.w) \, \overline{J(g)},
\end{equation}
where $K_{B_n}(z,w)$ denotes the Bergman kernel of the unit ball $B_n$ and $J(g)$ is the determinant of the complex jacobian matrix of $g$.
\end{proposition}

We point out that in the statement above, and in the rest of the paper, integrals are always with respect to the Lebesgue measure, which is therefore omitted from the notation.

\subsection{Main Results}
By an elementary manipulation, it is clear that the weighted estimate \eqref{Q_estimate} is equivalent to an unweighted estimate for the integral operator with kernel
\begin{equation}\label{kernel_KGp}
K_{G,p}(z,w) = |J\pi(z)|^{\frac{2}{p}-1}K_G(z,w)|J\pi(w)|^{1-\frac{2}{p}},
\end{equation}
where $J(\pi)$ is the determinant of the complex jacobian matrix of $\pi$.
Our main theorem provides a bound for the size of this kernel in dimension $2$.
\begin{theorem}\label{thm:main_result}
Let $p$ be in $(1,+\infty)$. There exists a constant $c>0$ such that 
\begin{equation}\label{eq:main_estimate}
|K_{G,p}(z,w)|\leqslant C\sum_{g\in G}|K_{B_2}(g.z,w)|
\end{equation}
for every $(z,w)$ in $B_2\times B_2$.
\end{theorem}

As a consequence we solve the $L^p$-regularity problem for two-dimensional Rudin ball quotients.
\begin{corollary}\label{cor:Lp_regularity}
Let $D$ be a Rudin ball quotient in $\C^2$. Then, there exists a constant $C>0$ such that 
\[
\|P_Df\|_{L^p}\leqslant c\|f\|_{L^p}\qquad \forall f\in L^2(D)\cap L^p(D).
\]
if and only if $p\in(1,+\infty)$.
\end{corollary}

We conclude the introduction with a brief explanation of why our results are established only in dimension $2$. Let us first fix some notation.

Since $G$ is a f.u.r.g., by definition it is generated by reflections.  
Given a reflection $r \in G$, let $Y_r$ denote its reflecting hyperplane,
\[
Y_r := \{ z \in \C^n \colon r.z = z \},
\]
and let $\mathcal R_G$ be the set of all reflecting hyperplanes of reflections in $G$. If $Y_r \in \mathcal{R}_G$, any unit vector $\rho$ orthogonal to $Y_r$ is called a \emph{root} of the reflection $r$.  
To each $Y \in \mathcal{R}_G$, we associate the subgroup
\[
G_Y := \{ g \in G \colon g.v = v \quad \forall v \in Y \},
\]
which consists of the identity together with all reflections $r \in G$ such that $Y_r = Y$.  
This subgroup $G_Y$ is cyclic of some finite order $m_Y$, which is called the \emph{multiplicity} of $Y$.

We recall the following classical result, which illustrates how the analytic description of a $G$-orbit map reflects the underlying geometry of the action of the group $G$.

\begin{proposition}[{\cite[Theorem 9.8]{lehrer_taylor}}]\label{prp_Jacobian}
Choose a root $\rho_Y$ for each reflecting hyperplane $Y$ of $G$. Then the Jacobian of a $G$-orbit map $\pi:\C^n \to \C^n$ is given by
\begin{equation}\label{eq:jacobian_pi}
J(\pi)(z) = c_\pi \prod_{Y \in \mathcal R_G} \left\langle z, \rho_Y \right\rangle^{m_Y - 1},
\end{equation}
where $c_\pi \in \C$ is a nonzero constant depending on the choice of generators, and $\left\langle \cdot, \cdot \right\rangle$ denotes the standard Hermitian product of $\C^n$.
\end{proposition}

It follows that the union of the reflecting hyperplanes of $G$ is precisely the zero set of $J(\pi)$.  
Thus, to estimate $K_{G,p}$, it is essential to understand how the zeros of $J(\pi)$ interact with $K_G$.  

In dimension $2$, we can take advantage of  the fact that two distinct hyperplanes intersect only at the origin. Consequently, if $z \neq 0$, two factors in \eqref{eq:jacobian_pi} cannot vanish simultaneously.  
In higher dimensions, the geometry is more intricate: distinct hyperplanes intersect in positive-dimensional subspaces, and a deeper understanding of the arrangement of the hyperplanes of a f.u.r.g. is required.

\section{Preliminary results}

Unless strictly necessary, we will state and prove our results for the $n$-dimensional ball $B_n$, specializing to the two-dimensional case only when it cannot be avoided. To obtain our main estimate, we decompose $\ol{B_n}\times\ol{B_n}$ in several regions, which depend on the geometry of the given f.u.r.g $G$. For each $g$ in $G$ and for each $\varepsilon>0$ set 
\[
\mathcal{\mathcal{U}}_{g}(\varepsilon)=\left\{ (z,w)\in\ol{B_n}\times\ol{B_n}:d(z,bB_n)+d(w,bB_n)+|g.z-w|<\varepsilon\right\}. 
\]
The next elementary lemma shows why the regions $\mathcal U_g(\varepsilon)$'s are useful to localize the singularities of the kernel
\begin{align*}
K_G(z,w)&= \frac{1}{|G|} \sum_{g \in G} K_{B_n}(z, g.w) \, \overline{J(g)}= \frac{1}{|G|} \sum_{g \in G} K_{B_n}(g.z, w) \, J(g)=\frac{1}{|G|}\sum_{g\in G}\frac{J(g)}{(1-\ipr{g.z,w})^{n+1}}.
\end{align*}

\begin{lemma}\label{lem:singularities}
    Let $\varepsilon>0$ and let $g$ in $G$ be fixed. Set 
    \[
\mathcal S_g(\varepsilon)=\left\{(z,w)\in\ol{B_n}\times\ol{B_n}: |1-\ipr{g.z,w}|<\varepsilon\right\}.
\]
Then, 
\[
\mathcal U_g(\varepsilon)\subseteq \mathcal S_g(3\varepsilon)\subseteq\mathcal U_g(12\varepsilon).
\]
\end{lemma}
\begin{proof}
Let $(z,w)$ be in $\mathcal U_g(\varepsilon)$. Then,
\begin{align*}
|1-\langle g.z,w\rangle|&\leqslant(1-|w|^2|)+|\langle g.z-w,w\rangle\leqslant2d(w,b B_2)+|g.z-w||w|< 3\varepsilon.
\end{align*}
Hence, the pair $(z,w)$ belongs to $\mathcal S_g(3\varepsilon)$. Now,
\begin{align*}
3\varepsilon&>|1-\langle g.z,w\rangle|\geqslant 1-|g.z||w|\geqslant 1-|w|=d(w,bB_2),
\end{align*}
where we used that $g$ is a unitary matrix. Similarly, we obtain
\[
3\varepsilon>|1-\langle g.z,w\rangle|\geqslant 1-|g.z|=1-|z|=d(z,bB_2).
\]
Then,
\[
|g.z-w|\leqslant (1-|g.z|)+(1-|w|)<6\varepsilon.
\]
In conclusion we get that $\mathcal S_g(3\varepsilon)$ is a subset of $\mathcal U_g(12\varepsilon)$. \end{proof}

It is crucial to understand how the regions $\mathcal U_g(\varepsilon)$ interact with one another as $g$ varies.

\begin{lemma}\label{lem:regions_intersection}
Let $n=2$. There exists $\varepsilon>0$ such that, for all $g,\ell\in G$ with
$\ell^{-1}g$ not a reflection, one has
\[
\mathcal U_{g}(\varepsilon)\cap\mathcal U_{\ell}(\varepsilon)=\emptyset.
\]
\end{lemma}

\begin{proof}
Let $g$ and $\ell$ such that $\ell^{-1}g$ is not a reflection be fixed. Assume that such $\varepsilon$ does not exist. Then, for every positive $\varepsilon_k$ there exists $(z_k,w_k)$ in $\mathcal U_{g}(\varepsilon_k)\cap \mathcal U_{\ell}(\varepsilon_k)$, that is,
\[
d(z_k,bB_2)+d(w_k, bB_2)+|g.z_k-w_k|<\varepsilon_k
\]
and
\[
d(z_k,bB_2)+d(w_k, bB_2)+|\ell.z_k-w_k|<\varepsilon_k.
\]
Let us consider a sequence $\{\varepsilon_k\}_k$ such that $\varepsilon_k\to 0^+$ as $k\to+\infty$ and the associated sequence $\{(z_k,w_k)\}_k$. Since the latter is a bounded sequence, it admits a limit point $(z,w)$.
From the above we deduce that, necessarily, it holds
\[
(z,w)\in b B_2\times bB_2
\]
and
\[
g.z=w=\ell.z\iff \ell^{-1}g.z=z.
\]
Thus, $z$ is a fixed point of $\ell^{-1}g$. However, since $\ell^{-1}g$ is not a reflection and we are in dimension $2$, its only fixed point is the origin, but $z$ belongs to $b B_2$. This is a contradiction. Therefore, there exists $\varepsilon>0$ such that $\mathcal U_{g}(\varepsilon)\cap\mathcal U_{\ell}(\varepsilon)=\emptyset$. The conclusion follows by letting vary $g$ and $\ell$ and observing that they vary in a finite group.
\end{proof}

The situation changes significantly if, on the contrary, $\ell^{-1}g$ is a reflection.

\begin{lemma}\label{lem:regions_intersection_2}                    Let $g,\ell$ be in $G$ and assume that $\ell^{-1}g$ is a reflection. Then, for every $\varepsilon>0$ it holds
\[
\mathcal U_g(\varepsilon)\cap\mathcal U_\ell(\varepsilon)\neq \emptyset.
\]
\end{lemma}

\begin{proof}
Let $\varepsilon>0$ be fixed. Set $Y_{\ell^{-1}g}$ to be the reflecting hyperplane of $\ell^{-1}g$. Let $z$ be in $Y_{\ell^{-1}g}$ such that $d(z,b B_n)<\varepsilon/2$ and consider the pair $(z, \ell.z)$. Then, since $\ell$ is unitary and $z\in Y_{\ell^{-1}g}$ ,
\begin{align*}
d(z,b B_n)+d(\ell.z,b B_n)+|g.z-\ell.z|&=d(z,b B_n)+d(z,b B_n)+|\ell^{-1}g.z-z|<\varepsilon,
\end{align*}
since $\ell^{-1}g.z=z$. Hence,$(z,\ell.z)$ is in $\mathcal U_{g}(\varepsilon)$. Similarly,
\[
d(z,b B_n)+d(\ell.z,b B_n)+|\ell.z-\ell.z|=d(z,bB_n)+d(z,bB_n)<\varepsilon,
\]
that is, $(z,\ell.z)$ is in $\mathcal U_{\ell}(\varepsilon)$. Hence, the conclusion follows.
\end{proof}

\begin{remark}\label{rkm:regions_intersection_3}The situation becomes more delicate when examining the intersection of three or more regions, so we restrict to the case $n=2$. In this setting, such an intersection may or may not be empty for sufficiently small $\varepsilon$, depending on the orders of the reflections involved.

Assume that for every $\varepsilon>0$,
\[
\mathcal U_g(\varepsilon)\cap \mathcal U_h(\varepsilon)\cap \mathcal U_\ell(\varepsilon)\neq \emptyset. 
\]
Arguing as in the previous lemmas, this would yield the existence of $(z,w)$ in $bB_2\times bB_2$ such that
\[
w=g.z=h.z=\ell.z.
\]
Thus, since we are in dimension $2$, $z$ must lie on the reflecting hyperplanes of the (possible) reflections $\ell^{-1}h$, $\ell^{-1}g$, and $g^{-1}h$. Since in $\mathbb C^2$ two distinct hyperplanes intersect only at the origin, this can only occur if these three reflections share the same reflecting hyperplane. However, if $G$ contains only reflections of order $2$, the intersection must necessarily be empty, as there cannot be two distinct reflections in $G$ with the same reflecting hyperplane. 

On the other hand, if $r$ in $G$ is a reflection of order $m\geq 3$, then $Y_{r^j}=Y_r$ for every $j=1,\dots,m-1$. Take $z\in Y_r$ such that $d(z,bB_2)<\frac{\varepsilon}{2}$ and consider the pair $(z,r^j.z)$. Then
\[
d(z,bB_2)+d(r^j.z,B_2)+|r^\ell.z-r^j.z|
= d(z,bB_2)+d(z,B_2)+|r^{\ell-j}.z-z|
\leq \varepsilon,
\]
so that, for every $\varepsilon>0$, the pair $(z,r^j.z)$ belongs to each region $\mathcal U_{r^\ell}(\varepsilon)$ for $\ell=0,\dots,m-1$.

\end{remark}

\smallskip

We conclude the section with a general observation on the symmetry properties of the kernel $K_{G,p}$ defined in \eqref{kernel_KGp}. A stronger statement was proved in \cite[Lemma~4.1]{DM}, but for our purposes a simpler fact will suffice. With the notation introduced in Proposition~\ref{prp_Jacobian}, we state the following classical result. 
\begin{proposition}[{\cite[Lemma~9.10]{lehrer_taylor}}]\label{prop:skew_invariance}
Let $P\in\C[z_1,\ldots, z_n]$ be a skew polynomial, i.e.,
\[
P(g.z)=J(g)^{-1}P(z)\qquad\forall g\in G.    
\]
Then $\prod_{Y\in \mathcal R_G}\langle z, e_Y\rangle^{m_Y-1}$ divides $P$.
\end{proposition}

Notice that, by Proposition \ref{prp_Jacobian}, the function $\prod_{Y\in \mathcal R_G}\langle z, e_Y\rangle^{m_Y-1}$ coincides,
up to a nonzero multiplicative constant, with the Jacobian $J(\pi)$ of any $G$-orbit map $\pi$.

Recall that 
\[
K_{G,p}(z,w)=|J\pi(z)|^{\frac{2}{p}-1}|J\pi(w)|^{1-\frac{2}{p}}K_G(z,w),
\]
where $\pi:\C^n\to \C^n$ is a $G$-orbit map and
\begin{align*}
K_{G}(z,w)&=\frac{1}{|G|}\sum_{g\in G}J(g)K_{B_n}(g.z,w)=\frac{1}{|G|}\frac{Q(z,w)}{\prod_{\ell\in G}(1-\langle \ell.z,w\rangle)^{n+1}}
\end{align*}
with 
\[
Q(z,w)=\sum_{m\in G}J(m)\prod_{\ell\neq m}(1-\ipr{\ell.z,w})^{n+1}.
\]
Notice that $P(z,w)=Q(z,\overline w)$ is a holomorphic polynomial as a function of $z$ and $w$. Also, $P(z,w)$ is a skew polynomial with  the respect to the natural action of the f.u.r.g $G\oplus \ol{G}$. Indeed, for $(g,\overline h)\in G\oplus \overline{G}$, we have
\begin{align*}
    P(g.z,\overline{h}.{w})&=Q(g.z, h.\overline w)=\sum_{m\in G}J(m)\prod_{\ell\neq m}(1-\langle h^{-1}\ell g.z,\overline w\rangle)^{n+1}\\
&=\sum_{m\in G}J(m)\prod_{\ell \neq h^{-1}mg}(1-\langle \ell.z,\overline w\rangle)^{n+1}=\sum_{hkg^{-1}\in G}J(h)J(k) \ol{J(g)}\prod_{\ell\neq k}(1-\langle \ell.z,\overline w\rangle)^{n+1}\\
&=J(h) \ol{J(g)}Q(z,\overline w)=J(g\oplus \overline h)^{-1}P(z,w).
\end{align*}
By Proposition \ref{prop:skew_invariance}, it follows that  $J\pi(z)\overline{J\pi(\overline w)}$ divides $P(z,w)$ for any $G$-orbit map $\pi$. Namely,
\begin{equation}\label{eq:skew_numerator}
Q(z,\overline w)=P(z,w)=J\pi(z)\overline{J\pi(\overline w)} M(z,w),
\end{equation}
where $M(z,w)$ is a holomorphic polynomial in $z$ and $w$. In conclusion, we get
\begin{align}
\begin{split}
\label{eq:skew_K_Gamma_p}
|K_{G,p}(z,w)|&=|J\pi(z)|^{\frac{2}{p}-1}|J\pi(w)|^{1-\frac{2}{p}}|K_G(z,w)|\\
&=|J\pi(z)|^{\frac{2}{p}}|J\pi(w)|^{2-\frac{2}{p}}\frac{|M(z,\overline w)|}{|\prod_{\ell\in G}(1-\langle \ell.z,w\rangle)^{n+1}|}.
\end{split}
\end{align}

Notice that the powers of the jacobian factors in \eqref{eq:skew_K_Gamma_p} are positive for every $p\in(1,\infty)$. The problem is now to understand how the zero sets of $M(z,\overline w)$ and of the denominator interact with each other. 

\section{Bounds on $K_{G,p}$}
In this section we bound the kernel $K_{G,p}$ by decomposing $\ol{B_n}\times\ol{B_n}$ using the regions $\mathcal U_g(\varepsilon)$'s.
From now on, $\varepsilon$ will always denote a sufficiently small positive constant, chosen so that Lemma \ref{lem:regions_intersection} applies.

\begin{lemma}\label{lem:easy_bounds}
Let $n=2$, let $G$ be a f.u.r.g. and fix $g$ in $G$. There exists a constant $c>0$ such that, for every $(z,w)$ in 
\[
\mathcal I_g(\varepsilon):=\mathcal{U}_g(\varepsilon) \setminus 
\bigg( 
  \bigcup_{\substack{\ell\in G, \\ \ell^{-1}g \textrm{ ref}}} 
  \Big(\mathcal{U}_g(\varepsilon) \cap \mathcal{U}_\ell(\varepsilon)\Big) 
\bigg),
\]
we have $|K_{G,p}(z,w)|\leqslant c |K_{B_2}(g.z,w)|$.
\end{lemma}


\begin{proof}
By \eqref{eq:skew_K_Gamma_p} we know that 
\[
|K_{G,p}(z,w)|=|J\pi(z)|^{\frac{2}{p}}|J\pi(w)|^{2-\frac{2}{p}}\frac{|M(z,\overline w)|}{|\prod_{\ell\in G}(1-\langle \ell.z,w\rangle)^{n+1}|}.
\]
Since $(z,w)\in\mathcal U_{g}(\varepsilon)$, by Lemma \ref{lem:singularities}, we have 
\[
|1-\langle g.z,w\rangle|<3\varepsilon.
\]
Let now $\ell$ be in $G$ such that $\ell^{-1}g$ is not a reflection. If $\varepsilon$ is sufficiently small so that $\mathcal U_g(12\varepsilon) \cap \mathcal U_{\ell}(12\varepsilon) = \varnothing$ (see Lemma \ref{lem:regions_intersection}), then 
\[
|1-\ipr{\ell.z,w}|\geqslant 3\varepsilon,\quad \forall(z,w)\in \mathcal I_g(\varepsilon),
\]
Otherwise, by Lemma \ref{lem:singularities}, we would have $(z,w)$ in $\mathcal U_g(12\varepsilon) \cap \mathcal U_{\ell}(12\varepsilon) \neq \varnothing$. 
If $\ell^{-1}g$ is a reflection, then, for every $(z,w)$ in $\mathcal I_g(\varepsilon)$, we have $|1-\ipr{\ell.z,w}|\geqslant \varepsilon/4$. Otherwise, again by Lemma \ref{lem:singularities}, $(z,w)$ would belong to $\mathcal U_\ell(\varepsilon)\cap\mathcal U_g(\varepsilon)$ but this cannot be the case by assumption since $(z,w)\in\mathcal I_g(\varepsilon)$.
In conclusion, 
\[
|1-\langle \ell.z,w\rangle|\geqslant\varepsilon/4,\quad \forall (z,w)\in\mathcal I_g(\varepsilon),\forall\ell\in  G,\ell\neq g.
\]
Therefore, for every $(z,w)$ in $\mathcal I_g(\varepsilon)$, it holds
\begin{align*}
|K_{G,p}(z,w)|\leqslant c\frac{|J\pi(z)|^{\frac2p}|J\pi(w)|^{2-\frac 2p}}{|\prod_{\ell\in G}(1-\langle \ell.z,w\rangle)^{n+1}|}\leqslant \frac{c}{|1-\langle g.z,w\rangle|^{n+1}}=c|K_{ B_n}(g.z,w)|.
\end{align*}
\end{proof}

We now turn to estimating the kernel $K_{G,p}$ on the region $\mathcal U_{g}(\varepsilon)\cap\mathcal U_{\ell}(\varepsilon)$ when $\ell^{-1}g$ is a reflection. This requires some intermediate steps and auxiliary results. The next two lemmas are adaptations of \cite[Theorem A.1]{DM}, originally proved for the case of a group generated by a single reflection of order $2$. We show that the same strategy can be implemented for reflections of arbitrary finite order, with suitable modifications. We include the proofs to make the argument self-contained. For clarity, we restrict to the two-dimensional setting, pointing out that an analogous reasoning applies in higher dimensions as well.

\begin{lemma}\label{lem:easy_identity}
 Let $G\subseteq U(2)$ be the f.u.r.g. 
\[
G=\left\{\begin{pmatrix}
    e^{\frac{2\pi ij}{m}} & 0 \\
    0 & 1
\end{pmatrix}, \,j=0,\ldots,m-1\right\}.
\]
Set
\begin{equation}\label{eq:region_G}
\mathcal G=\big\{(z,w)\in\mathbb B_2\times\mathbb B_2: |z_1\overline{w_1}|\geqslant \textstyle\frac12 |1-z_2\overline{w_2}|\big\}.
\end{equation}
If $(z,w)\in\mathcal G^c$, we have the identity
\begin{align}
K_{G}(z,w)=\frac{1}{m}\sum_{j=0}^{m-1}\frac{e^{\frac{2\pi i j}{m}}}{(1-e^{\frac{2\pi i j}{m}}z_1\overline{w}_1 - z_2\overline{w}_2)^3}=\frac{(z_1\overline{w}_1)^{m-1}}{(1 - z_2 \overline{w}_2)^{m+2}}A(z,w),
\end{align}
where $A(z,w)$ is a bounded function on $\mathcal G^c$.
\end{lemma}

\begin{proof}
The proof follows from a straightforward computation. For $(z,w)$ in $\mathcal G^c$ we have
\[
\frac{1}{(1-e^{\frac{2\pi ij}{m}}z_1\overline{w_1}-z_2\ol{w_2})^3}=\frac{1}{(1-z_2\ol{w_2})^3}\sum_{k=0}^{+\infty}\frac{(k+2)(k+1)}{2}\bigg(\frac{e^{\frac{2\pi i j}{m}}z_1\ol{w_1}}{1-z_2\ol{w_2}}\bigg)^k.
\]
So that,
\begin{align*}
K_{G}(z,w)&=\frac{1}{m(1-z_2\ol{w_2})^3}\sum_{j=0 }^{m-1}e^{\frac{2\pi ij}{m}}\sum_{k=0}^{+\infty}\frac{(k+2)(k+1)}{2}\bigg(\frac{e^{\frac{2\pi i j}{m}}z_1\ol{w_1}}{1-z_2\ol{w_2}}\bigg)^k\\
&=\frac{1}{m(1-z_2\ol{w_2})^3}\sum_{k=0}^{+\infty}\frac{(k+2)(k+1)}{2}\bigg(\frac{z_1\ol{w_1}}{1-z_2\ol{w_2}}\bigg)^k\sum_{j=0}^{m-1}    e^{\frac{2\pi ij}{m}(k+1)}.
\end{align*}
The inner sum is nonzero if and only if $k\equiv -1\mod m$. Hence,

\begin{align*}
&K_{G}(z,w)=\frac{1}{(1-z_2\ol{w_2})^3}\sum_{\ell=1}^{+\infty}\frac{((\ell m -1)+2)((\ell m -1)+1)}{2}\bigg(\frac{z_1\ol{w_1}}{1-z_2\ol{w_2}}\bigg)^{\ell m-1}\\
&=\frac{1}{(1-z_2\ol{w_2})^3}\bigg(\frac{z_1\ol{w_1}}{1-z_2\ol{w_2}}\bigg)^{m-1}\sum_{\ell=1}^{+\infty}\frac{(m(\ell+1)+1)(m(\ell+1))}{2}\bigg(\frac{z_1\ol{w_1}}{1-z_2\ol{w_2}}\bigg)^{\ell m}.
\end{align*}
The last series converges absolutely since in $\mathcal G^c$ we have $|z_1\ol w_1|/|1-z_2\ol w_2|<1/2$.
\end{proof}

\begin{lemma}\label{lem:cyclic_group}
Let $G\subseteq U(2)$ be a f.u.r.g. generated by a single reflection of any finite order. Then, there exists a constant $c>0$ such that 
\[
|K_{G,p}(z,w)|\leqslant c\sum_{g\in G}|K_{B_2}(g.z,w)|, \quad\forall (z,w)\in B_2\times B_2.
\]
\end{lemma}
\begin{proof} First note that if the inequality above holds for a f.u.r.g. $G\subseteq U(2)$, it also holds for each conjugate f.u.r.g. $H$ as well. Suppose that $G=uHu^{-1}=\{uhu^{-1}: h\in H\}$ for a $u\in U(2)$, then formula 
\eqref{averaged_kernel} shows that $K_H(z,w)=K_G(uz,uw).$ The orbit maps $\pi_G$ and $\pi_H$ can be taken to be related 
by $\pi_H=\pi_G\circ u,$ so that $\abs{J\pi_H(z)}= \abs{J\pi_G(u(z))},$  and consequently $K_{H,p}(z,w)=K_{G,p}(uz,uw).$
Therefore assuming that the result holds for the group $G$ we have
\begin{align*}
    \abs{K_{H,p}(z,w)}= K_{G,p}(uz,uw)\leqslant c\sum_{g\in G}|K_{B_2}(g.uz,uw)|&= c\sum_{g\in G}|K_{B_2}(u^{-1}g.uz,w)|\\&=c\sum_{h\in H}|K_{B_2}(h.z,w)|,
\end{align*}
so that the result also holds for the conjugate group $H$ (with the same constant).  Therefore, after a unitary conjugation, we may assume that
\[
G=\left\{
\begin{pmatrix}
    e^{\frac{2\pi i j}{m}} & 0 \\
    0 & 1
\end{pmatrix} : j=0,\ldots,m-1
\right\}.
\]
For this group one has
\[
K_{G,p}(z,w)
= \frac{1}{m}\bigg(\frac{|z_1|}{|w_1|}\bigg)^{(m-1)(\frac{2}{p}-1)}
   \sum_{j=0 }^{m-1}\frac{e^{\frac{2\pi i j}{m}}}{\big(1-e^{\frac{2\pi i j}{m}}z_1\overline{w_1}-z_2\ol{w_2}\big)^3}.
\]
Let $\mathcal G$ be the region defined in \eqref{eq:region_G}.  
If $(z,w)\in\mathcal G^c$, then by Lemma \ref{lem:easy_identity} we obtain
\[
K_{G,p}(z,w)
= \bigg(\frac{|z_1|}{|w_1|}\bigg)^{(m-1)(\frac{2}{p}-1)}
   \frac{(z_1\ol w_1)^{m-1}}{(1-z_2\ol w_2)^{m+2}} A(z,w).
\]

Therefore,
\[
|K_{G,p}(z,w)|
  \leqslant \frac{|z_1|^{(m-1)\frac{2}{p}}\,|w_1|^{(m-1)(2-\frac2p)}}
                 {|1-z_2\ol{w_2}|^{m+2}}.
\]
Observe that
\[
|z_1|^2 \leqslant 1-|z_2|^2 \leqslant 2|1-z_2\ol w_2|,
\]
and a similar estimate holds for $|w_1|$. Hence,
\[
|K_{G,p}(z,w)|\leqslant \frac{1}{|1-z_2\ol w_2|^3}.
\]
Since $(z,w)\in\mathcal G^c$, we further have
\[
|1-z_1\ol w_1-z_2\ol w_2|\leqslant \tfrac{3}{2}|1-z_2\ol w_2|,
\]
and thus
\begin{equation}\label{eq:Gc_bound}
|K_{G,p}(z,w)|\leqslant c\,|K_{B_2}(z,w)|.
\end{equation}

On the other hand, if $(z,w)\in\mathcal G$, then
\[
\frac{|z_1|}{|w_1|}
   = \frac{|z_1|^2}{|w_1||z_1|}
   \leqslant 2\frac{|z_1|^2}{|1-z_2\ol w_2|}
   \leqslant 2\frac{1-|z_2|^2}{1-|z_2||\ol w_2|}
   \leqslant 4,
\]
and the same bound holds for $\tfrac{|w_1|}{|z_1|}$. Consequently,
\begin{equation}\label{eq:G_bound}
|K_{G,p}(z,w)| \leqslant c\sum_{g\in G}|K_{B_2}(g.z,w)|.
\end{equation}
The claim follows by combining \eqref{eq:Gc_bound} and \eqref{eq:G_bound}.
\end{proof}

Our next goal is to prove the following major step toward the proof of Theorem \ref{thm:main_result}.
\begin{proposition}
\label{lem:specific}
    Let $G$ be a f.u.r.g..     There exists a constant $c>0$ such that, for every reflection $r\in \mathcal{R}_G$,
\[
|K_{G,p}(z,w)|\leqslant c\sum_{\ell\in G}|K_{B_2}(\ell.z,w)|, \qquad \forall (z,w)\in \mathcal U_{r}(\varepsilon)\cap\mathcal U_{\id}(\varepsilon).
\]
Here $\id$ denotes the identity element of the group $G$.
\end{proposition}

We need two elementary lemmas. 

\begin{lemma}\label{lem:jacobian_hyperplane} Let $G\subseteq U(n)$ be a f.u.r.g.. Then there are constants $C_1$ and $C_2$ 
such that, for each reflection $r\in \mathcal{R}_G$ and 
for each $z\in \cx^n$,
\begin{equation}\label{eq:easy_identity}
C_1 \abs{\ipr{z,\rho}}\leq \abs{r.z-z}\leq C_2\abs{\langle z,\rho\rangle},
\end{equation}
 where $\rho$ is a root of the reflection $r$.

\end{lemma}
\begin{proof}
Denote by $\Pi$ the orthogonal projection from $\cx^n$ onto $\rho^\perp$, the reflecting hyperplane of $r$. Then
\[r.z=\Pi z+ e^{i\theta_r} \ipr{z,\rho}\rho\]
where $0< \theta_r\leq \pi.$ Since $z=\Pi z+  \ipr{z,\rho}\rho$ we see that 
\[ z-r.z=(1-e^{i\theta_r}) \ipr{z,\rho}\rho,\]
so that 
\[ \abs{z-r.z}^2=2 \sin^2 \frac{\theta_r}{2} \cdot \abs{\ipr{z,\rho}}^2, \]
so that the result follows, with $C_1=\sqrt{2} \min\limits_{r\in \mathcal{R}_G} \sin \frac{\theta_r}{2}$ and $C_2=\sqrt{2}\max\limits_{r\in \mathcal{R}_G} \sin \frac{\theta_r}{2}$

\end{proof}

\begin{remark}\label{rmk:change_coordinates}
Observe that the unitary change of coordinates
\[
\begin{cases}
    \zeta = h.z, \\
    \omega = w
\end{cases}
\]
ensures that $(z,w)\in \mathcal U_g(\varepsilon)\cap \mathcal U_h(\varepsilon)$ if and only if $(\zeta,\omega)\in \mathcal U_{gh^{-1}}(\varepsilon)\cap\mathcal U_{\id}(\varepsilon)$, where $\id$ denotes the identity element of the group. Moreover, if $h^{-1}g$ is a reflection, then so is $gh^{-1}$, since
\[
gh^{-1} = h (h^{-1}g) h^{-1},
\]
that is, $gh^{-1}$ is obtained by conjugation of a reflection. Consequently, it suffices to focus on regions of the form $\mathcal U_r(\varepsilon)\cap\mathcal U_\id(\varepsilon)$ with $r$ a reflection. Estimates on the more general regions $\mathcal U_g(\varepsilon)\cap\mathcal U_h(\varepsilon)$ with $h^{-1}g$ a reflection will then follow by a unitary change of coordinates.
\end{remark}

\begin{lemma}\label{lem:estimate_reflecting_hyperplane}
Let $G$ be a  f.u.r.g.. Then there exists a constant $c > 0$ such that, for each $r\in \mathcal{R}_G$,
\[\mathcal U_r(\varepsilon) \cap \mathcal U_{\id}(\varepsilon) \subseteq \{(z,w): \abs{\langle z, \rho \rangle} \leq c \varepsilon, \abs{\langle w, \rho \rangle}\leq c \varepsilon \}.    \]
where $\rho$ is a root of $r$.  

\end{lemma}
    
\begin{proof} Let $(z,w)\in \mathcal U_r(\varepsilon) \cap \mathcal U_{\id}(\varepsilon) $.
By \eqref{eq:easy_identity} and the definition of $\mathcal U_r(\varepsilon)\cap\mathcal U_\id(\varepsilon)$ we immediately get 
\begin{equation}\label{eq:easy_identity_2}
\abs{\ipr{z,\rho}}\leq\frac{1}{C_1}\abs{r.z-z}\leqslant
\frac{1}{C_1}\left(|r.z-w|+|z-w|\right)<\frac{2\varepsilon}{C_1}.
\end{equation}
Notice that $(z,w)\in \mathcal U_r(\varepsilon) \cap \mathcal U_{\id}(\varepsilon)$ if and only if $(w,z)\in \mathcal U_{r^{-1}}(\varepsilon) \cap \mathcal U_{\id}(\varepsilon)$. Then \eqref{eq:easy_identity_2} shows that we must have $\abs{\langle w, \rho \rangle}\leq c \varepsilon$ as well.

\end{proof}

We are ready to prove Proposition \ref{lem:specific}. Here it is essential that we are working in dimension $n=2$.
\begin{proof}[Proof Proposition \ref{lem:specific}]
Define $Y_{r}$ to be the reflecting hyperplane of $r$. Namely, assume that
\[
Y_r=\Big\{z: \langle z,\rho_r\rangle=0\Big\},\quad |\rho_r|=1.
\]
By Lemma \ref{lem:estimate_reflecting_hyperplane}, if $(z,w)\in \mathcal U_r(\varepsilon)\cap\mathcal U_\id(\varepsilon)$, then both $z$ and $w$ are uniformly close to the hyperplane $Y_r$. Hence, this and the fact that two hyperplanes in $\mathbb C^2$ only intersect at the origin, assure that for $(z,w)$ in $\mathcal U_r(\varepsilon)\cap\mathcal U_\id(\varepsilon)$, the linear forms $\ipr{z, \rho_g},\ipr{w, \rho_g}$ are bounded from below for every $g\neq r$, $g$ reflection in $G$ . Here $\rho_g$ denotes a root of the hyperplane $Y_g$. Hence, recalling that $J(\pi)$ is given by \eqref{eq:jacobian_pi},  we  get  that, for every $(z,w)\in \mathcal U_{r}(\varepsilon)\cap\mathcal U_{\id}(\varepsilon)$, 
\begin{equation}\label{eq:KGp_intermediate}
|K_{G,p}(z,w)|\leqslant c |\langle z,\rho_r\rangle|^{(\frac 2p-1)(m_r-1)}|\langle w,\rho_r\rangle|^{(1-\frac 2p)(m_r-1)}|K_{G}(z,w)|,
\end{equation}
where $m_r$ is the order of the reflection $r$.

Let now $H=\langle r\rangle$ be the f.u.r.g generated by the single reflection $r$ and consider the lateral decomposition of $G$ with respect to this subgroup $H$, that is,
\[
G=\bigsqcup_{g\in G/H}gH.
\]
Then,
\begin{align}
\begin{split}
K_{G}(z,w)&=\frac{1}{|G|}\sum_{g\in G/H}J(g) \sum_{\ell\in H}J(\ell) K_{B_2}(g\ell.z,w)=\frac{|H|}{|G|}\sum_{g\in G/H}J(g) K_{H}(z,g^*.w)\\
&=c\bigg(K_H(z,w)+\sum_{\substack{g \in G/ H \\ g \ne \id}}J(g) K_{H}(z,g^*.w)\bigg)=:c\bigg(K_H(z,w)+B(z,w)\bigg).
\end{split}
\end{align}
By Lemma \ref{lem:cyclic_group} we know that
\begin{equation}\label{eq:K_H_easy}
|\langle z,\rho_r\rangle|^{(\frac 2p-1)(m_r-1)}|\langle w,\rho_r\rangle|^{(1-\frac 2p)(m_r-1)}|K_{H}(z,w)|\leqslant c|K_{ B_2}(z,w)|
\end{equation}
for every $(z,w)$ in $B_2\times B_2$. Let us now focus on
\[
B(z,w)=\sum_{\substack{g \in  G/ H \\ g \ne \id}}J(g) K_{H}(z,g^*.w)=\sum_{\substack{g \in G / H \\ g \ne \id}}J(g) \frac{\sum_{\ell\in H}J(\ell) \prod_{q\neq \ell}(1-\ipr{q.z, g^*.w})^3}{\prod_{\ell \in H}(1-\ipr{\ell.z, g^*.w})^3}.
\]

From \eqref{eq:skew_numerator} and Proposition \ref{prop:skew_invariance} we obtain that
\begin{align*}
B(z,w)&=\sum_{\substack{g \in G / H \\ g \ne \id}}J(g) \frac{(\ipr{z,\rho_r}\ol{\ipr{g^*.w,\rho_r}})^{m_{r}-1}P_g(z,w)}{\prod_{\ell \in H}(1-\ipr{\ell.z, g^*.w})^3},
\end{align*}
where $P_g(z,w)$ is a holomorphic polynomial in $z$ and antiholomorphic in $w$. We conclude that
\begin{equation}\label{eq:equation_B}
B(z,w)=\ipr{z,\rho_r}^{m_r-1}\frac{P_{H}(z,w)}{\prod_{\substack{g \in G \backslash H}}(1-\ipr{g.z,w})^3},
\end{equation}
where $P_{H}(z,w)$ is a holomorphic polynomial as a function of $z$ and an antiholomorphic polynomial as a function of $w$. We use the notation $P_H$ with the subscript $H$ to emphasize that such polynomial is determined by the subgroup $H$.

Recall also that 
\[
B(z,w)= \sum_{g\in G}\frac{J(g)}{(1-\ipr{g.z,w})^3}-\sum_{\ell\in H}\frac{J(\ell)}{(1-\ipr{\ell.z,w})^3}.
\]
Hence, $B(z,w)=\overline{B(w,z)}$ since the right-hand side of the above identity satisfies the same property. Also, setting
\[
L(z,w)=\prod\limits_{g \in G\backslash H} (1-\ipr{g.z,w})^3,
\]
we notice that $L(z,w)=\overline{L(w,z)}$ as well. Indeed,
\begin{align*}
L(z,w)&=\prod_{g\in G\backslash H} (1-\ipr{g.z,w})^3=\overline{\prod_{g\in G\backslash H} (1-\ipr{w,g.z})^3}\\ &=\overline{\prod_{g\in G\backslash H} (1-\ipr{g^*.w,z})^3}=\overline{\prod_{g\in G\backslash H} (1-\ipr{g.w,z})^3}\\
&=\overline{L(w,z)},
\end{align*}
where the second-last inequality holds true since  $g\notin H\iff g^*=g^{-1}\notin H$ being $H$ a subgroup.
Therefore,
\begin{align*}
B(z,w)=\overline{B(w,z)}&\iff\ipr{z,\rho_r}^{m_r-1}\frac{P_H(z,w)}{L(z,w)}=\overline{\ipr{w,\rho_r}}^{m_r-1}\frac{\overline{P_H(w,z)}}{\overline{L(w,z)}}\\
&\iff  \ipr{z,\rho_r}^{m_r-1}P_H(z,w)=\overline{\ipr{w,\rho_r}}^{m_r-1}\,\overline{P_H(w,z)}.
\end{align*}
Now, we see that whenever the antiholomorphic linear form $\overline{\ipr{w,\rho_r}}$ vanishes the antiholomorphic polynomial (as a function of $w$) $P_{H}(z,w)$ vanishes as well.  Therefore in the polynomial ring 
$\mathbb{C}[z_1,\dots,z_n, \ol{w_1},\dots, \ol{w_n}]$ in $2n$ variables, the polynomial $p$ given by
$p(z_1,\dots, \dots, z_n, \ol{w_1},\dots, \ol{w_n})= \overline{\ipr{w,\rho_r}}$ divides $P_{H}(z,w)$: this follows 
by either directly from the Nullstellensatz, or by a linear change of variables in the polynomial ring in which $\overline{\ipr{w,\rho_r}}$ is taken to be one of the new variables. Repeating this argument for each $\rho_r$,  we conclude that
\[
\ipr{z,\rho_r}^{m_r-1}P_H(z,w)=\ipr{z,\rho_r}^{m_r-1}\overline{\ipr{w,\rho_r}}^{m_r-1}Q_H(z,w),
\]
where $Q_H$ is a holomorphic polynomial as a function of $z$  and antiholomorphic as a function of $w$. In conclusion,
\[
B(z,w)= \frac{\ipr{z,\rho_r}^{m_r-1}\overline{\ipr{w,\rho_r}}^{m_r-1}Q_H(z,w)}{L(z,w)}.
\]
Hence,
\begin{align}\label{eq:estimate_B}
\begin{split}
    \left( \frac{|\ipr{z,\rho_r}|}{\ipr{w,\rho_r}} \right)^{\left(\tfrac{2}{p}-1\right)(m_r-1)}
    |B(z,w)|
    &\leqslant 
    |\ipr{z,\rho_r}|^{\tfrac{2}{p}(m_r-1)}
    |\ipr{w,\rho_r}|^{\left(2-\tfrac{2}{p}\right)(m_r-1)}
    \frac{|Q_H(z,w)|}{|L(z,w)|}\leqslant c.
\end{split}
\end{align}

The last estimate follows from the fact that $Q_H(z,w)$ is a polynomial, hence bounded in $\ol{B_2}\times\ol{B_2}$, and that  in $\mathcal U_r(\varepsilon)\cap \mathcal U_e(\varepsilon)$ we have $1/|L(z,w)|$ bounded as well. Indeed, for $g\in G\backslash H$, in order to have
\begin{align*}
|1-\ipr{z,g^*.w}|&=|1-\ipr{g.z,w}|<\varepsilon/4
\end{align*}
we would need $(z,w)$ in $\mathcal U_g(\varepsilon)$ by Lemma \ref{lem:singularities}. However, we already have that $(z,w)\in\mathcal U_r(\varepsilon)\cap\mathcal U_\id(\varepsilon)$ and, by Remark \ref{rkm:regions_intersection_3}, we know that $\mathcal U_{g}(\varepsilon)\cap\mathcal U_h(\varepsilon)\cap\mathcal U_\id(\varepsilon)$ is necessarily empty unless $g\in H$.
By \eqref{eq:K_H_easy} and \eqref{eq:estimate_B} we conclude the proof.

\end{proof}
We now prove the same estimate as in Proposition \ref{lem:specific}, but in more general regions.

\begin{proposition}\label{prop:kernel_hard_bounds}
Let $G$ be a f.u.r.g.. There exists $c>0$ such that
\[
|K_{G,p}(z,w)|\leqslant c\sum_{\ell\in G}|K_{B_2}(\ell.z,w)|,
\]
for every $(z,w)$ in $\mathcal U_g(\varepsilon)\cap\mathcal U_h(\varepsilon)$ with $ h^{-1}g$ a reflection.

\end{proposition}

\begin{proof}
We have that $(z,w)\in\mathcal U_g(\varepsilon)\cap U_h(\varepsilon)$ if and only if $(\zeta,\omega)\in\mathcal U_{gh^{-1}}(\varepsilon)\cap\mathcal U_{e}(\varepsilon)$ where 
\[
\begin{cases}
    &\zeta=h.z\\
    &\omega=w.
\end{cases}
\]
Recall that $gh^{-1}$ is a reflection whenever $h^{-1}g$ is (Remark \ref{rmk:change_coordinates}). Then,
\begin{align*}
|K_{G,p}(z,w)|&=|K_{G,p}(h^{-1}.\zeta, \omega)|=|J(\pi)(h^{-1}.\zeta)|^{\frac 2p-1}|J(\pi)(\omega)|^{1-\frac 2p}\frac{1}{|G|}\bigg|\sum_{g\in G}\frac{J(g)}{(1-\ipr{gh^{-1}.\zeta,\omega})^3}\bigg|\\
&=|J(\pi)(\zeta)|^{\frac 2p-1}|J(\pi)(\omega)|^{1-\frac 2p}\frac{|J(h)|}{|G|}\bigg|\sum_{g\in G}\frac{J(g) }{(1-\ipr{g.\zeta,\omega})^3}\bigg|\\
&=|K_{G,p}(\zeta,\omega)|.
\end{align*}
Notice that we used the identity $|J\pi(h^{-1}.\zeta|=|J\pi(\zeta)|$. This follows from the fact that if $\rho_r$ is a root of the hyperplane $Y_r$ associated to the reflection $r$, then $h.\rho_r$ is a root for the reflecting hyperplane associated to the reflection $hrh^{-1}$. 
 Thus, by Lemma \ref{lem:specific}, we obtain
 \[
 |K_{G,p}(z,w)|=|K_{G,p}(\zeta,\omega)|\leqslant c\sum_{g\in G}|K_{B_2}(g.\zeta,\omega)|\leqslant c\sum_{g\in G}|K_{B_2}(g h.z,w)|\leqslant c\sum_{g\in G}|K_{B_2}(g.z,w)|.
 \]
\end{proof}

The proof of Theorem \ref{thm:main_result}
is now immediate.  

\begin{proof}[Proof of Theorem \ref{thm:main_result}]
By Lemma \ref{lem:easy_bounds} and Proposition \ref{prop:kernel_hard_bounds}, we already have the desired estimate for $|K_{G,p}|$ on the region $\bigcup_{g\in G} \mathcal U_g(\varepsilon)$. On the complement of this region in $B_2 \times B_2$, we can directly use the identity \eqref{eq:skew_K_Gamma_p} to see that $|K_{G,p}|$ remains bounded. 

Combining these observations, we thus obtain the uniform estimate
\[
|K_{G,p}(z,w)| \leq c \sum_{g \in G} |K_{ B_2}(g.z, w)|, \qquad \forall (z,w) \in B_2 \times B_2,
\]
as claimed.
\end{proof}
At last, we solve the $L^p$-regularity problem for every Rudin ball quotients in $\C^2$ as promised. 
\begin{proof}[Proof of Corollary \ref{cor:Lp_regularity}]
If $p\in(1,+\infty)$ the result follows from the main estimate \eqref{eq:main_estimate} and the well-known $L^p$-regularity of the positive Bergman projection of the unit ball, that is, the integral operator associated to the kernel $|K_{B_2}(z,w)|$ (see, e.g., \cite[Chapter 7]{rudin_ftub}). The unboundedness for $p=1,+\infty$ follows from adapting a known  argument for the unboundedness of the Bergman projection of the unit ball; see \cite{rudin_ftub} and \cite[Theorem 1.16]{DM} 
\end{proof}

\bibliographystyle{alpha}
\bibliography{Rudin_ball_quotients_bib}

\end{document}